\documentclass[11pt,a4paper]{article}
\usepackage[T1]{fontenc}
\usepackage[utf8]{inputenc}
\usepackage{lmodern}
\usepackage{amsmath,amssymb,amsthm,mathtools}
\usepackage[margin=2.55cm]{geometry}
\usepackage{microtype}
\usepackage{enumitem}
\usepackage[colorlinks=true,linkcolor=black,citecolor=black,urlcolor=blue]{hyperref}
\usepackage[nameinlink,noabbrev]{cleveref}

\setlist{nosep}

\newtheorem{theorem}{Theorem}[section]
\newtheorem{lemma}[theorem]{Lemma}

\theoremstyle{remark}

\newcommand{\E}{\mathbb E}
\newcommand{\cL}{\mathcal L}
\newcommand{\cH}{\mathcal H}
\newcommand{\conv}{\operatorname{conv}}
\newcommand{\supp}{\operatorname{supp}}
\newcommand{\Span}{\operatorname{span}}
\newcommand{\norm}[1]{\left\|#1\right\|}
\newcommand{\ip}[2]{\left\langle #1,#2\right\rangle}
\newcommand{\signs}{\{-1,1\}}

\title{Large signed sums of unit vectors: the first linearly dependent case}
\author{Dami\'an Pinasco}
\date{}

\begin{document}
\maketitle

\begin{abstract}
We address a problem on large signed sums of unit vectors that arose in work of Brugger, Fiedler, Gonz\'alez Merino and Kirschbaum and was later formulated in its present form by Ambrus and Nietert. Given $d+1$ unit vectors $u_1,\ldots,u_{d+1}$ in $\mathbb R^d$, with $d\ge2$, the problem asks for the smallest possible value of
\[
 \max_{\varepsilon_i=\pm1}
 \left\|\sum_{i=1}^{d+1}\varepsilon_i u_i\right\|.
\]
We prove that this value is $\sqrt{d+2}$. We also determine all equality cases: up to independent sign changes and orthogonal transformations, they consist of the vertices of a centered regular simplex of positive even dimension together with an orthonormal basis of its orthogonal complement.
\end{abstract}



\section{Introduction}

Given unit vectors $u_1,\ldots,u_n\in\mathbb R^d$, the elementary average
\[
 \frac{1}{2^n}\sum_{\varepsilon\in\signs^n}
 \left\|\sum_{i=1}^n\varepsilon_i u_i\right\|^2=n
\]
shows that some signed sum has norm at least $\sqrt n$. When $n\le d$, this bound is sharp, as witnessed by an orthonormal system. Thus $n=d+1$ is the first case in which linear dependence is forced, and the natural question is whether that dependence forces a larger signed sum.

Large signed sums may equivalently be viewed as an $\ell_1$-polarization problem, since the elementary identity
\[
 \max_{\varepsilon\in\signs^n}
 \left\|\sum_{i=1}^n\varepsilon_i u_i\right\|
 =\max_{v\in S^{d-1}}\sum_{i=1}^n |\ip{v}{u_i}|
\]
holds for every system of vectors $u_1,\ldots,u_n$. This point of view was used by Ambrus and Nietert to place the signed-sum problem in the framework of spherical polarization \cite{AmbrusNietert2019}. The same discrete formulation is also used in the subsequent study of minimax spherical designs by Fu, Wang and Yan \cite{FuWangYan2023}.

The $d+1$ case emerged from a more general selection problem considered by Brugger, Fiedler, Gonz\'alez Merino and Kirschbaum in connection with the behavior of the circumradius under Minkowski addition \cite{BruggerEtAl2018}. Specializing their Conjecture~4.2 to pairs $\{\pm u_i\}$ leads to the signed-sum problem considered here. After further study, including a communication from A. Polyanskii, the conjecture took the form relevant to the present paper: the sharp bound should be $\sqrt{d+2}$ for every $d\ge2$, and equality should occur precisely for an even-dimensional regular simplex together with orthogonal unit directions. Ambrus and Nietert subsequently stated this form explicitly as Conjecture~1 in the signed-sum setting \cite{AmbrusNietert2019}.

Ambrus and Gonz\'alez Merino later studied the broader signed subset-sum problem, in which one selects $k$ vectors among $n$ unit vectors and assigns signs to maximize the norm of their sum. For the $d+1$ problem considered here, their Theorem~6 proves the conjectured bound under the additional assumptions that $d$ is even and
\[
 \sum_{i=1}^{d+1}u_i=0
\]
\cite{AmbrusGonzalez2021}. A related centered problem has a zonotope formulation: Jo\'os and L\'angi minimize the circumradius, at fixed mean width, among zonotopes with $d+1$ generators in centered canonical form; see their Theorem~5 \cite{JoosLangi2023}. The case $d=2$ is already contained in the planar result of Brugger et al.\ \cite[Theorem~1.3]{BruggerEtAl2018}. The next dimension was settled by Fu, Wang and Yan: their Proposition~2.6 shows that the minimax spherical design for four vectors in $\mathbb R^3$ has value $\sqrt5$ and is of triangular-pyramid type, namely an equilateral triangle in a plane together with a perpendicular unit vector \cite{FuWangYan2023}. Thus their result gives the conjectured value $\sqrt5$ in dimension three, together with the predicted equality configuration.

More recently, Ambrus and Grundbacher have developed the large-signed-sum/polarization correspondence for general finite-dimensional Minkowski spaces \cite{AmbrusGrundbacher2026}. The present paper concerns the Euclidean $d+1$ problem and proves the conjecture for every $d\ge2$, including the complete equality classification.

We prove the following statement.

\begin{theorem}\label{thm:main}
Let $d\ge2$ and let $u_1,\ldots,u_{d+1}\in S^{d-1}$. Then
\[
 \max_{\varepsilon\in\signs^{d+1}}
 \left\|\sum_{i=1}^{d+1}\varepsilon_i u_i\right\|^2\ge d+2.
\]
Equality holds if and only if, after independent sign changes and a permutation of the vectors, there is a positive even integer $k\le d$ and an orthogonal decomposition
\[
 \mathbb R^d=E\oplus E^\perp,\qquad \dim E=k,
\]
such that $u_1,\ldots,u_{k+1}$ are the unit vertices of a centered regular simplex in $E$, while $u_{k+2},\ldots,u_{d+1}$ form an orthonormal basis of $E^\perp$.
\end{theorem}

The proof combines induction with a rigidity argument for minimizing configurations. If some $d$ of the vectors are linearly dependent, the lower bound follows from the induction hypothesis. Otherwise, the space of linear relations among the $d+1$ vectors is one-dimensional. A convex balance condition for the maximizing sign vectors then yields explicit pairwise correlations, and a weighted variance identity gives the sharp bound. In the equality case, the same identity forces a second linear dependence with cubic coefficients, from which the original dependence is seen to be centered. A parity argument and an average over a suitable sign layer then force a regular simplex. The remaining equality cases are obtained inductively by adjoining orthogonal unit directions.

\section{The model equality configurations}\label{sec:models}

We first verify the configurations appearing in the equality statement.

Let $v_1,\ldots,v_{k+1}$ be the unit vertices of a centered regular simplex in a $k$-dimensional Euclidean space. Thus
\[
 \sum_{i=1}^{k+1}v_i=0,
 \qquad
 \ip{v_i}{v_j}=-\frac1k\quad(i\ne j).
\]
For every $\varepsilon\in\signs^{k+1}$,
\begin{equation}
 \left\|\sum_{i=1}^{k+1}\varepsilon_i v_i\right\|^2
 =(k+1)-\frac{2}{k}\sum_{i<j}\varepsilon_i\varepsilon_j
 =\frac{(k+1)^2-\left(\sum_{i=1}^{k+1}\varepsilon_i\right)^2}{k}.
 \label{eq:simplexformula}
\end{equation}
If $k$ is even, then $k+1$ is odd, so the minimum possible value of
$\left|\sum_{i=1}^{k+1}\varepsilon_i\right|$ is $1$. Hence
\[
 \max_{\varepsilon}
 \left\|\sum_{i=1}^{k+1}\varepsilon_i v_i\right\|^2=k+2,
\]
and the maximizing sign vectors are exactly those for which
$\left|\sum_{i=1}^{k+1}\varepsilon_i\right|=1$.

Now let $e_1,\ldots,e_{d-k}$ be an orthonormal basis of a space orthogonal to the simplex. Every signed sum in the orthonormal block has squared norm $d-k$, independently of the signs. Therefore the direct sum configuration has maximum squared norm
\[
 (k+2)+(d-k)=d+2.
\]
Thus every configuration listed in \cref{thm:main} is an equality configuration.

\section{The sharp lower bound}\label{sec:lower}

For a fixed system $u_1,\ldots,u_{d+1}$ write
\[
 S_\varepsilon:=\sum_{i=1}^{d+1}\varepsilon_i u_i,
 \qquad
 M:=\max_{\varepsilon\in\signs^{d+1}}\norm{S_\varepsilon}^2.
\]
We prove the lower bound by induction on $d$. The one-dimensional estimate needed as the base case is immediate: two unit vectors in $\mathbb R$ admit a signed sum of norm $2$, so $M=4\ge3$.

Suppose first that some $d$ of the vectors are linearly dependent; after reordering, take $u_1,\ldots,u_d$. Their span has dimension at most $d-1$, so the induction hypothesis, applied after an isometric embedding into $\mathbb R^{d-1}$, gives signs $\varepsilon_1,\ldots,\varepsilon_d$ such that
\[
 S:=\sum_{i=1}^d\varepsilon_i u_i,
 \qquad \norm S^2\ge d+1.
\]
Choosing the sign of $u_{d+1}$ favorably,
\[
 \max_{\delta=\pm1}\norm{S+\delta u_{d+1}}^2
 =\norm S^2+1+2|\ip{S}{u_{d+1}}|
 \ge d+2.
\]
Hence only the case in which every $d$-subfamily is linearly independent requires further work.

Since there are $d+1$ vectors in $\mathbb R^d$ and every $d$-subfamily is independent, the space of linear relations among them is one-dimensional. After independent sign changes and a normalization we may write
\begin{equation}\label{eq:dependence}
 \sum_{i=1}^{d+1}a_i u_i=0,
 \qquad a_i>0,
 \qquad \sum_{i=1}^{d+1}a_i^2=1.
\end{equation}
All $a_i$ are nonzero, and every linear relation among the $u_i$ is a scalar multiple of \eqref{eq:dependence}.

We now consider a minimizing configuration. Related stationarity conditions obtained from tangent perturbations appear in Fu, Wang and Yan \cite[Lemma~2.5]{FuWangYan2023}; we use the following convex form. Consider $M$ as a continuous function on $(S^{d-1})^{d+1}$ and choose a configuration at which it attains its minimum. If some $d$ vectors of this minimizing configuration are linearly dependent, the preceding argument already gives $M\ge d+2$. We may therefore assume that every $d$-subfamily is linearly independent, and hence that the minimizing configuration satisfies \eqref{eq:dependence}. Write
\begin{equation}\label{eq:rdefpaper}
 M=d+1+r.
\end{equation}
The uniform average over all sign vectors is $d+1$, so $r\ge0$.

Let
\[
 \cL:=\{\varepsilon\in\signs^{d+1}:\norm{S_\varepsilon}^2=M\}
\]
be the set of maximizing sign vectors. For $h_i\perp u_i$, consider the normalized curves
\[
 u_i(t)=\frac{u_i+t h_i}{\norm{u_i+t h_i}}.
\]
Since $u_i'(0)=h_i$, differentiation gives
\[
 \frac{d}{dt}\bigg|_{t=0}
 \left\|\sum_{i=1}^{d+1}\varepsilon_i u_i(t)\right\|^2
 =2\sum_{i=1}^{d+1}\varepsilon_i\ip{S_\varepsilon}{h_i}.
\]
Let $P_i$ denote the orthogonal projection onto $u_i^\perp$ and set
\[
 q_\varepsilon
 :=(\varepsilon_1P_1S_\varepsilon,\ldots,
    \varepsilon_{d+1}P_{d+1}S_\varepsilon)
 \in H:=\bigoplus_{i=1}^{d+1}u_i^\perp.
\]
Set
\[
 C:=\conv\{q_\varepsilon:\varepsilon\in\cL\}.
\]
We claim that
\begin{equation}\label{eq:convexbalancepaper}
 0\in C.
\end{equation}
Suppose, to the contrary, that $0\notin C$, and let $q_0\in C$ be a point of minimal norm. Since $C$ is compact and convex, such a point exists. For every $q\in C$ and $0<s<1$,
\[
 q_0+s(q-q_0)\in C,
\]
and hence the minimality of $q_0$ gives
\[
 \norm{q_0+s(q-q_0)}^2\ge \norm{q_0}^2.
\]
After expanding, dividing by $s$ and letting $s\downarrow0$, we obtain
\[
 \ip{q_0}{q}\ge \norm{q_0}^2
 \qquad(q\in C).
\]
Take $h=-q_0\in H$. For every maximizing sign vector $\varepsilon\in\cL$, the derivative at $t=0$ of $\norm{\sum_{i=1}^{d+1}\varepsilon_i u_i(t)}^2$ is then at most $-2\norm{q_0}^2<0$. Since $\cL$ is finite, all the corresponding squared norms are below $M$ for the same sufficiently small positive $t$. For every remaining sign vector $\varepsilon\notin\cL$ we have $\norm{S_\varepsilon}^2<M$; by finiteness and continuity, these strict inequalities persist for all sufficiently small $t$. This contradicts the minimality of the configuration and proves \eqref{eq:convexbalancepaper}.

Thus there are numbers $p_\varepsilon\ge0$, $\varepsilon\in\cL$, such that
\[
 \sum_{\varepsilon\in\cL}p_\varepsilon=1,
 \qquad
 \sum_{\varepsilon\in\cL}p_\varepsilon q_\varepsilon=0.
\]
Such a representation need not be unique. We choose one and fix it throughout the argument, and write $\E$ for expectation with respect to the corresponding probability distribution on $\cL$. Since
\[
 \E q_\varepsilon=0\qquad\text{in }H,
\]
the $i$th coordinate gives
\[
 \E(\varepsilon_i P_iS_\varepsilon)=0.
\]
Equivalently,
\[
 P_i\,\E(\varepsilon_iS_\varepsilon)=0.
\]
Thus $\E(\varepsilon_iS_\varepsilon)$ is parallel to $u_i$, and hence, for some scalar $\lambda_i$,
\begin{equation}\label{eq:balancepaper}
 \E(\varepsilon_iS_\varepsilon)=\lambda_i u_i,
 \qquad 1\le i\le d+1.
\end{equation}

We next record the pairwise correlations forced by the balance condition. Set
\[
 c_{ij}:=\E(\varepsilon_i\varepsilon_j).
\]
Expanding \eqref{eq:balancepaper},
\[
 \sum_{j=1}^{d+1}(c_{ij}-\lambda_i\delta_{ij})u_j=0.
\]
The space of linear relations among the $u_j$ is one-dimensional, generated by \eqref{eq:dependence}. Hence, for each $i$, there is a scalar $\mu_i$ such that
\[
 c_{ij}-\lambda_i\delta_{ij}=\mu_i a_j
 \qquad(1\le j\le d+1).
\]
For $i\ne j$, symmetry of $(c_{ij})$ gives $\mu_i a_j=\mu_j a_i$. Since every $a_i$ is nonzero, there is a scalar $\kappa$ such that $\mu_i=\kappa a_i$ for all $i$. Therefore
\[
 c_{ij}=\lambda_i\delta_{ij}+\kappa a_i a_j.
\]
Since $c_{ii}=1$,
\[
 \lambda_i=1-\kappa a_i^2.
\]
On the other hand, the distribution is supported on maximizing sign vectors, so
\[
 d+1+r
 =\E\norm{S_\varepsilon}^2
 =\sum_{i=1}^{d+1}\ip{\E(\varepsilon_iS_\varepsilon)}{u_i}
 =\sum_{i=1}^{d+1}\lambda_i
 =d+1-\kappa.
\]
Therefore $\kappa=-r$, and
\begin{equation}\label{eq:correlationspaper}
 c_{ij}=-r a_i a_j\quad(i\ne j),
 \qquad
 \lambda_i=1+r a_i^2.
\end{equation}

For $\varepsilon\in\cL$, let $\varepsilon^{(i)}$ be obtained from $\varepsilon$ by replacing $\varepsilon_i$ with $-\varepsilon_i$, and define
\[
 \rho_i(\varepsilon):=\varepsilon_i\ip{S_\varepsilon}{u_i}-1.
\]
Then
\[
 \norm{S_{\varepsilon^{(i)}}}^2
 =\norm{S_\varepsilon}^2-4\rho_i(\varepsilon)
 =M-4\rho_i(\varepsilon).
\]
Since $\varepsilon$ is maximizing, $\rho_i(\varepsilon)\ge0$. This elementary observation is closely related to Bang's lemma; see \cite{Bang1951} and \cite[Section~4]{AmbrusNietert2019}.
Summing over $i$ and using $\norm{S_\varepsilon}^2=M=d+1+r$, while \eqref{eq:balancepaper} and \eqref{eq:correlationspaper} give the expectation, we obtain
\begin{equation}\label{eq:rhosumpaper}
 \sum_{i=1}^{d+1}\rho_i(\varepsilon)
 =\norm{S_\varepsilon}^2-(d+1)=r,
 \qquad
 \E\rho_i(\varepsilon)
 =\lambda_i-1
 =r a_i^2.
\end{equation}
Taking the inner product of \eqref{eq:dependence} with $S_\varepsilon$ and writing
\[
 b_\varepsilon:=\sum_{i=1}^{d+1} a_i\varepsilon_i,
\]
we obtain
\begin{equation}\label{eq:bidentitypaper}
 b_\varepsilon=-\sum_{i=1}^{d+1} a_i\varepsilon_i\rho_i(\varepsilon).
\end{equation}
A direct expansion gives
\begin{align}
 r\sum_{i=1}^{d+1} a_i^2\rho_i(\varepsilon)-b_\varepsilon^2
 &=\frac12\sum_{i=1}^{d+1}\sum_{j=1}^{d+1}\rho_i(\varepsilon)\rho_j(\varepsilon)
   (a_i\varepsilon_i-a_j\varepsilon_j)^2
 \ge0.
 \label{eq:deficitpaper}
\end{align}
Taking expectations in \eqref{eq:deficitpaper}, the first term is
\[
 \E\left[r\sum_{i=1}^{d+1} a_i^2\rho_i(\varepsilon)\right]
 =r^2\sum_{i=1}^{d+1} a_i^4.
\]
Using \eqref{eq:correlationspaper},
\begin{align*}
 \E b_\varepsilon^2
 &=\sum_{i=1}^{d+1}a_i^2+2\sum_{1\le i<j\le d+1}a_i a_jc_{ij}\\
 &=1-r+r\sum_{i=1}^{d+1}a_i^4.
\end{align*}
Consequently
\[
 1-r+r\sum_{i=1}^{d+1} a_i^4\le r^2\sum_{i=1}^{d+1} a_i^4,
\]
or
\[
 (1-r)\left(1+r\sum_{i=1}^{d+1} a_i^4\right)\le0.
\]
Since $r\ge0$, the second factor is positive. Thus $r\ge1$, and \eqref{eq:rdefpaper} gives $M\ge d+2$. The induction is complete.

\section{Equality when every \texorpdfstring{$d$}{d}-subfamily is independent}\label{sec:uniqueeq}

Assume now that $M=d+2$ and that every $d$-subfamily is linearly independent. By \cref{sec:lower}, such a configuration is a global minimizer, and all the preceding identities apply with $r=1$.

We first extract the additional dependence forced by equality. At $r=1$, the expectation of the nonnegative quantity in \eqref{eq:deficitpaper} is zero. Hence it vanishes for every $\varepsilon\in\supp p$. Using \eqref{eq:rhosumpaper} and \eqref{eq:bidentitypaper},

\begin{align*}
 \sum_{i=1}^{d+1}\rho_i(\varepsilon)(a_i\varepsilon_i+b_\varepsilon)^2
 &=\sum_{i=1}^{d+1} a_i^2\rho_i(\varepsilon)
   +2b_\varepsilon\sum_{i=1}^{d+1} a_i\varepsilon_i\rho_i(\varepsilon)
   +b_\varepsilon^2\sum_{i=1}^{d+1}\rho_i(\varepsilon)\\
 &=\sum_{i=1}^{d+1} a_i^2\rho_i(\varepsilon)-b_\varepsilon^2.
\end{align*}
Thus the vanishing of the deficit may be written as
\begin{equation}\label{eq:deficitr1}
 \sum_{i=1}^{d+1}\rho_i(\varepsilon)(a_i\varepsilon_i+b_\varepsilon)^2=0
 \qquad(\varepsilon\in\supp p).
\end{equation}
Hence
\begin{equation}\label{eq:pointequalitypaper}
 \rho_i(\varepsilon)(a_i\varepsilon_i+b_\varepsilon)=0
 \qquad
 (\varepsilon\in\supp p,\ 1\le i\le d+1).
\end{equation}

Define
\[
 W:=\E(b_\varepsilon S_\varepsilon).
\]
Since $c_{ij}=-a_i a_j$ for $i\ne j$,
\begin{align*}
 \E(b_\varepsilon\varepsilon_i)
 &=a_i+\sum_{\substack{j=1\\j\ne i}}^{d+1}a_jc_{ij}
 =a_i-a_i\sum_{\substack{j=1\\j\ne i}}^{d+1}a_j^2
 =a_i^3.
\end{align*}
Thus
\begin{equation}\label{eq:Wcubic}
 W=\sum_{i=1}^{d+1} a_i^3u_i.
\end{equation}
On the other hand, \eqref{eq:pointequalitypaper} gives
$b_\varepsilon\varepsilon_i\rho_i(\varepsilon)=-a_i\rho_i(\varepsilon)$ on $\supp p$. Since $\E\rho_i(\varepsilon)=a_i^2$,
\begin{align*}
 \ip{W}{u_i}
 &=\E\bigl[b_\varepsilon\varepsilon_i(1+\rho_i(\varepsilon))\bigr]\\
 &=a_i^3-a_i\E\rho_i(\varepsilon)=0.
\end{align*}
The $u_i$ span $\mathbb R^d$, hence $W=0$. By \eqref{eq:Wcubic},
\[
 \sum_{i=1}^{d+1} a_i^3u_i=0.
\]
The relation space is one-dimensional, so $(a_i^3)_i$ is proportional to $(a_i)_i$. Because all $a_i$ are positive, all $a_i^2$ are equal, and the normalization in \eqref{eq:dependence} gives
\begin{equation}\label{eq:equalapaper}
 a_i=\frac1{\sqrt{d+1}}
 \qquad\text{for all }i.
\end{equation}
Returning to the relation \eqref{eq:dependence}, we obtain
\begin{equation}\label{eq:centeredpaper}
 \sum_{i=1}^{d+1}u_i=0.
\end{equation}

The same pointwise equality has a parity consequence. For every $\varepsilon\in\supp p$, \eqref{eq:rhosumpaper} gives $\sum_{i=1}^{d+1}\rho_i(\varepsilon)=1$, so some $\rho_i(\varepsilon)$ is positive. For that index, \eqref{eq:pointequalitypaper} and \eqref{eq:equalapaper} give
\[
 \sum_{j=1}^{d+1}\varepsilon_j=-\varepsilon_i\in\{-1,1\}.
\]
The sum of $d+1$ signs has the same parity as $d+1$. Hence $d+1$ is odd and
\begin{equation}\label{eq:devenpaper}
 d\ \text{is even}.
\end{equation}

Write $d=2m$. We now use the centered relation \eqref{eq:centeredpaper}. Consider the sign layer
\[
 \cH:=\left\{\varepsilon\in\signs^{d+1}:\sum_{i=1}^{d+1}\varepsilon_i=1\right\}.
\]
Let $\E_{\cH}$ denote expectation with respect to the uniform distribution on $\cH$. Related layer averages occur in the centered zonotope argument of Jo\'os and L\'angi \cite[Theorem~5]{JoosLangi2023}. Permutation invariance gives a common value $\alpha$ for
$\E_{\cH}(\varepsilon_i\varepsilon_j)$ when $i\ne j$. Since $\sum_{i=1}^{d+1}\varepsilon_i=1$ on $\cH$,
\[
 1
 =\E_{\cH}\left(\sum_{i=1}^{d+1}\varepsilon_i\right)^2
 =(d+1)+d(d+1)\alpha,
\]
so
\begin{equation}\label{eq:layercorrpaper}
 \alpha=-\frac1{d+1}.
\end{equation}
Let $g_{ij}=\ip{u_i}{u_j}$. From \eqref{eq:centeredpaper} we have
\begin{equation}\label{eq:gramtotalpaper}
 \sum_{i=1}^{d+1}\sum_{\substack{j=1\\j\ne i}}^{d+1}g_{ij}=-(d+1).
\end{equation}
Therefore
\begin{align*}
 \E_{\cH}\norm{S_\varepsilon}^2
 &=(d+1)
   -\frac1{d+1}\sum_{i=1}^{d+1}\sum_{\substack{j=1\\j\ne i}}^{d+1}g_{ij}\\
 &=d+2.
\end{align*}
Since $M=d+2$, every term in this average is at most $d+2$. As the average itself equals $d+2$, equality holds throughout. Consequently,
\begin{equation}\label{eq:alllayermaxpaper}
 \norm{S_\varepsilon}^2=d+2
 \qquad\text{for every }\varepsilon\in\cH.
\end{equation}

It remains to identify the centered configuration. A sign vector in $\cH$ has exactly $m$ negative signs. If
$A=\{i:\varepsilon_i=-1\}$, then $|A|=m$, and using \eqref{eq:centeredpaper},
\[
 S_\varepsilon=-2\sum_{i\in A}u_i.
\]
Thus \eqref{eq:alllayermaxpaper} says that
\begin{equation}\label{eq:msubsets}
 \left\|\sum_{i\in A}u_i\right\|^2=\frac{m+1}{2}
 \qquad\text{for every }A\subset\{1,\ldots,2m+1\},\ |A|=m.
\end{equation}

For $m=1$, \eqref{eq:equalapaper} already says that three unit vectors sum to zero, hence they form a regular triangle. Assume $m\ge2$. Fix $i\ne j$ and let
$T\subset\{1,\ldots,2m+1\}\setminus\{i,j\}$ have cardinality $m-1$. Comparing in \eqref{eq:msubsets} the sets $T\cup\{i\}$ and $T\cup\{j\}$ gives
\begin{equation}\label{eq:replacepaper}
 \sum_{\ell\in T}
 \bigl(\ip{u_\ell}{u_i}-\ip{u_\ell}{u_j}\bigr)=0.
\end{equation}
For fixed $i,j$ put
$x_\ell=\ip{u_\ell}{u_i}-\ip{u_\ell}{u_j}$ for $\ell\ne i,j$.
Every sum of $m-1$ of the $x_\ell$ is zero. Comparing two such subsets which differ only by replacing an index $r$ with an index $s$ shows that $x_r=x_s$. Hence all the $x_\ell$ are equal, and because $m-1>0$ their common value must be zero. Therefore
\begin{equation}\label{eq:triplepaper}
 \ip{u_\ell}{u_i}=\ip{u_\ell}{u_j}
 \qquad\text{whenever }i,j,\ell\text{ are distinct}.
\end{equation}
For each fixed $\ell$, all off-diagonal entries in the $\ell$th row of the Gram matrix are equal. Symmetry then shows that all off-diagonal entries have one common value, say $\beta$. Taking the inner product of \eqref{eq:centeredpaper} with $u_j$ gives
\[
 1+d\beta=0.
\]
Thus
\[
 \ip{u_i}{u_j}=-\frac1d\qquad(i\ne j),
\]
and the configuration is a centered regular $d$-simplex. Together with \eqref{eq:devenpaper}, this completes the equality classification when every $d$-subfamily is independent.

\section{The remaining equality cases}\label{sec:inductiveeq}

We now complete the classification by induction on $d$. We first record a property of the model configurations from \cref{sec:models}.

\begin{lemma}\label{lem:spanmaxpaper}
For a configuration consisting of a centered regular simplex of positive even dimension and an orthonormal basis of its orthogonal complement, the maximizing signed sums span the whole ambient space generated by the configuration.
\end{lemma}

\begin{proof}
Let $v_1,\ldots,v_{k+1}$ be the simplex block, with $k$ even, and let $e_1,\ldots,e_r$ be the orthonormal block. By \eqref{eq:simplexformula}, a sign pattern is maximizing on the simplex precisely when the simplex signs sum to $\pm1$, while the orthonormal signs are arbitrary.

Fix a maximizing simplex pattern and all orthonormal signs except that of $e_j$. Flipping the sign of $e_j$ gives two maximizing total sums whose difference is $2e_j$. Thus every $e_j$ lies in the span of the maximizing sums.

For $i\ne j$, choose a simplex sign pattern with total sign sum $1$, with $\varepsilon_i=1$ and $\varepsilon_j=-1$. Such a pattern exists because the remaining $k-1$ signs can be chosen to have sum $1$. Interchanging the signs at $i$ and $j$ preserves the layer and hence maximality. The two corresponding signed sums differ by $2(v_i-v_j)$. Finally,
\[
 v_i=\frac1{k+1}\sum_{j=1}^{k+1}(v_i-v_j),
\]
because $\sum_{j=1}^{k+1}v_j=0$. Hence the differences $v_i-v_j$ span the simplex space. Together with the orthonormal block, the maximizing sums span the full space.
\end{proof}

We prove the necessity in \cref{thm:main}. The case $d=2$ follows from \cref{sec:uniqueeq}; indeed, if two of the three vectors were dependent, their signed sum could have norm $2$, and adding the third vector with favorable sign would give squared norm at least $5$, contradicting equality $M=4$.

Assume now that the classification is known in dimension $d-1$, and let $u_1,\ldots,u_{d+1}\in S^{d-1}$ satisfy $M=d+2$. If every $d$-subfamily is independent, \cref{sec:uniqueeq} gives a regular simplex and $d$ is even.

Otherwise, after reordering, $u_1,\ldots,u_d$ are linearly dependent. Let
\[
 E:=\Span\{u_1,\ldots,u_d\},
 \qquad
 L:=\max_{\varepsilon\in\signs^d}
 \left\|\sum_{i=1}^d\varepsilon_i u_i\right\|^2.
\]
Viewing $E$ isometrically inside $\mathbb R^{d-1}$ and using the already proved lower bound in dimension $d-1$ gives $L\ge d+1$. Let
$S=\sum_{i=1}^d\varepsilon_i u_i$ be any maximizing signed sum of the block. Then
\begin{align*}
 d+2=M
 &\ge \max_{\delta=\pm1}\norm{S+\delta u_{d+1}}^2\\
 &=L+1+2|\ip{S}{u_{d+1}}|\\
 &\ge d+2.
\end{align*}
Thus
\begin{equation}\label{eq:inductioneqpaper}
 L=d+1,
 \qquad
 \ip{S}{u_{d+1}}=0
 \quad\text{for every maximizing sum }S\text{ of the block}.
\end{equation}
The block is therefore an equality configuration for the $(d-1)$-dimensional problem, after the above isometric embedding. By the induction hypothesis it is a model configuration from \cref{sec:models}. Every such model in $\mathbb R^{d-1}$ spans the whole ambient space: its regular simplex spans its simplex subspace and the remaining vectors form an orthonormal basis of the orthogonal complement. Hence the embedded copy of $E$ must be all of $\mathbb R^{d-1}$, and in particular $\dim E=d-1$. By \cref{lem:spanmaxpaper}, the maximizing sums of the block span $E$. The second part of \eqref{eq:inductioneqpaper} therefore gives
\[
 u_{d+1}\perp E.
\]
Hence $u_{d+1}$ adds one orthogonal unit direction to the orthonormal block of the lower-dimensional equality configuration. This is exactly the family described in \cref{thm:main} and completes the proof.

\section*{Declaration on the use of generative AI}

During the preparation of this manuscript, the author used ChatGPT (OpenAI) for exploratory mathematical discussion, checking calculations and exposition, bibliographic assistance, and language editing. All mathematical statements, proofs, and references in the final manuscript were independently checked by the author, who takes full responsibility for its contents.

\bigskip
\noindent
Universidad Torcuato Di Tella, Departamento de Matem\'atica y Estad\'istica, and CONICET,\\
Av. Figueroa Alcorta 7350, C1428BCW Buenos Aires, Argentina.

\noindent
\textit{E-mail address:} \href{mailto:dpinasco@utdt.edu}{\texttt{dpinasco@utdt.edu}}

\end{document}